\documentclass[11pt,oneside]{article}
\usepackage[english]{babel}
\usepackage{amssymb}
\usepackage{amsthm}
\usepackage{amsmath}
\usepackage{a4wide}
\usepackage{esvect}
\usepackage{hyperref}
\usepackage{cleveref}
\usepackage{thmtools}
\usepackage{thm-restate}
\usepackage{setspace}
\usepackage{verbatim}
\usepackage[dvips]{graphicx}
\usepackage{t1enc}
\usepackage{color}
\usepackage{pgf,tikz}
\usepackage{mathrsfs}
\usetikzlibrary{arrows}
\usepackage{geometry}
\usepackage{todonotes}
\newtheorem{thm}{Theorem}[section]%[chapter]
\newtheorem{claim}[thm]{Claim}
\newtheorem{lem}[thm]{Lemma}
\newtheorem{lemma}[thm]{Lemma}
\newtheorem{conjecture}[thm]{Conjecture}
\newtheorem{proposition}[thm]{Proposition}

\newtheorem{problem}[thm]{Problem}

\newtheorem{fact}[thm]{Fact}

\newtheorem{corollary}[thm]{Corollary}
\newtheorem{question}[thm]{Question}

\DeclareMathOperator{\pp}{pp}
\DeclareMathOperator{\diam}{diam}

\newcommand{\eps}{\varepsilon}
\newcommand{\dist}{\text{dist}}
\newcommand{\mM}{\mathcal{M}}
\newcommand{\mP}{\mathcal{P}}
\newcommand{\mR}{\mathcal{R}}
\newcommand{\mE}{\mathcal{E}}

\DeclareMathOperator{\tp}{tp}

\newcommand{\ignore}[1]{}

\title{On the maximum degree of path-pairable planar graphs}
 \author{Ant\'onio Gir\~ao \thanks{Department of Pure Mathematics and Mathematical Statistics, University of Cambridge, Cambridge, UK; \texttt{A.Girao@dpmms.cam.ac.uk}} 
\and
G\'abor M\'esz\'aros \thanks{Department of Mathematical Sciences, The University of Memphis, Memphis, Tennessee;
\texttt{gmszaros@memphis.edu}} 
\and
Kamil Popielarz \thanks{Department of Mathematical Sciences, The University of Memphis, Memphis, Tennessee; \texttt{kamil.popielarz@gmail.com}} 
\and 
Richard Snyder \thanks{Department of Mathematical Sciences, The University of Memphis, Memphis, Tennessee;
\texttt{rjsnyder23@gmail.com }} }

\begin{document}
\maketitle
\singlespace
\begin{abstract}
A graph is \emph{path-pairable} if for any pairing of its vertices there exist edge-disjoint paths joining the vertices in each pair. We investigate the behaviour of the maximum degree in path-pairable planar graphs. We show that any $n$-vertex path-pairable planar graph must contain a vertex of degree linear in $n$.
\end{abstract}
\onehalfspace
\section{Introduction}

We are interested in {\it path-pairability}, a graph theoretical notion that emerged from a practical networking problem. This notion was introduced by Csaba, Faudree, Gy\'arf\'as, Lehel, and Shelp \cite{CS}, and further studied by Faudree, Gy\'arf\'as, and Lehel \cite{mpp,F,pp}, and by Kubicka, Kubicki and Lehel \cite{grid}. Given a fixed integer $k$ and a simple undirected graph $G$ on at least $2k$ vertices, we say that $G$  is {\it $k$-path-pairable} if, for any pair of disjoint sets of distinct vertices $\{x_1,\dots,x_k\}$ and $\{y_1,\dots,y_k\}$ of $G$, there exist $k$ edge-disjoint paths $P_1,P_2,\dots,P_k$, such that $P_i$ is a path from $x_i$ to $y_i$, $1\leq i\leq k$. The concept of $k$-path pairability is closely related to the notions of $k$-\textit{linkedness} and $k$-\textit{weak-linkedness}. A graph is said to be $k$-(weakly)linked if for any choice $\{s_1,\ldots,s_k, t_1,\ldots ,t_k \}$ of $2k$ vertices (not necessarily distinct) there are vertex(edge) internally disjoint paths $P_1,\ldots, P_k$ with $P_i$ joining $s_i$ to $t_i$, $1\leq i\leq k$. While any $k$-(weakly)linked graph is $(2k-1)$-vertex connected ($k$-edge connected), the same need not hold for $k$-path-pairable graphs. Observe that the \textit{stars} $S_{2k}$ ($k\geq 1$) are $k$-path-pairable and yet have very low edge density and edge connectivity. 
On the other hand, a result of  Bollob\'as and Thomason \cite{BollobasThomason} shows that a  $2k$-connected graph with a lower bound on the edge density implies that $G$ is $k$-linked. A similar theorem of Hirata, Kubota and Saito \cite{HIRATASAITOKUBOTA} 
states that a $(2k+1)$-edge connected graph is $(k+2)$-weakly-linked for $k \geq 2$.

A $k$-path-pairable graph on $2k$ vertices is simply said to be {\it path-pairable}.
Some of the most central questions in the study of path-pairable graphs concern determining the behaviour of their maximum degree. It is fairly easy to construct path-pairable graphs on $n$ vertices ($n$ even) with maximum degree linear in $n$. For example, complete graphs $K_{2n}$ and complete bipartite graphs $K_{m,n}$ are path-pairable for all choices of $m,n\in\mathbb{N}$ with $m+n$ even, $m\neq 2,n\neq 2$.

It is slightly more challenging to construct an infinite family of path-pairable graphs where the maximum degree grows sublinearly. We 
shall now describe such a family. Let $K_t$ be the complete graph on $t$ vertices and let $K_t^q$ be constructed from $K_t$ by attaching $q -1$ leaves to the original vertices of $K_t$. This family was introduced by Csaba, Faudree, Gy\'arf\'as, and Lehel~\cite{CS}, who also proved that $K_t^q$ is path-pairable as long as $t\cdot q$ is even and $q\leq \left\lfloor\frac{t}{3+2\sqrt{2}}\right\rfloor$. The bound on $q$ has been recently improved to $\approx \frac{1}{3}t$ \cite{tpc}. 
Observe that $n = |V(K_t^q)| = t\cdot q$ and $\Delta(K_t^q) = t + q - 2 = O(\sqrt{n})$ when $q = \Omega(t)$. 
Additional path-pairable constructions with maximum degree $c \sqrt{n}$ can be found in \cite{grid} and \cite{gm_pp}.
%It is conjectured that hypercubes in odd dimensions are path-pairable. If the conjecture is true it further improves the constant of the upper bound on $\Delta$. 
%\begin{conjecture}\label{hypercube}
%The $d$-dimensional hypercube $Q_d$ is path-pairable for $d$ odd. 
%\end{conjecture}

The following result due to Faudree \cite{F}, shows that the maximum degree of a path-pairable graph has to grow with the order of the graph.
\begin{thm}\label{deltaLowerBound}
If $G$ is path-pairable on $n$ vertices with maximum degree $\Delta$, then $n\leq 2\Delta^\Delta$.
\end{thm}

Letting $\Delta_{\min}(n) := \min\{\Delta(G): G \text{ is a path-pairable graph on $n$ vertices}\}$,
this result is equivalent to 
\[
	\Delta_{\min}(n) \geq c_1\frac{\log n}{\log\log n},
\]
for some constant $c_1$. To date, the best known upper bound on $\Delta_\text{min}(n)$ is due to  Gy\H ori, Mezei, and M\'esz\'aros, exhibiting a path-pairable graph with maximum degree $\Delta \approx 5.5\cdot\log n$~\cite{ntp}. In summary, we have the following general asymptotic bounds on $\Delta_\text{min}(n)$:
\[c_1\frac{\log n}{\log\log n} \leq \Delta_\text{min}(n) \leq c_2\log n.\] 

We are interested in determining the behaviour  of the maximum degree in path-pairable \emph{planar} graphs. Let us define $\Delta^{p}_\text{min}(n)$ to be
\[
\min \{\Delta(G):G \text{ is a path-pairable planar graph on } n \text{ vertices}\}.	\]

Faudree, Gy\'arf\'as and Lehel \cite{pp} proved that a path-pairable graph without a $K_{1,n}$ must have at least $3n/2-\log n-c$ edges, for some absolute constant $c$. Although planar graphs might have considerably more than $(3/2-o(1))n$ edges we wished to determine whether planarity would be enough to force a vertex of linear degree.

We first note that a simple application of the Planar Separator Theorem of Lipton and Tarjan~\cite{LiptonTarjan} shows that every path-pairable planar graph on $n$ vertices must contain a vertex of degree at least $c\sqrt{n}$. Indeed, if $G$ is such a graph, then the Separator Theorem allows us to partition $V(G)$ into three sets $S$, $A$, $B$, where $|S| = O(\sqrt{n})$, $|A| \leq |B| \leq 2n/3$, and there are no edges between $A$ and $B$. Now, while path-pairable graphs $G$ need not be highly connected or edge connected, they must satisfy certain connectivity-like conditions. More precisely, they must satisfy the $\emph{cut-condition}$: for every subset $X \subset V(G)$ of size at most $n/2$, there are at least $|X|$ edges between $X$ and $V(G)\setminus X$. Note that the cut-condition is not sufficient to guarantee path-pairability; see \cite{gm_pp} for additional details. Accordingly, since $n/4 < |A| < n/2$ and there are no edges between $A$ and $B$, the cut-condition implies that there are at least $|A|$ edges between $A$ and $S$. We therefore obtain a vertex in $S$ of degree $\Omega(\sqrt{n})$.

Our main theorem, which we state below, shows that we can do much better than this. Namely, every path-pairable planar graph must have a vertex of \emph{linear} degree.

\begin{restatable}{thm}{thmPPP}\label{thm:PPP}
There exists $c \geq 10^{-10^{10}}$ such that if $G$ is a path-pairable planar graph on $n$ vertices then $\Delta(G) \geq cn$.
\end{restatable}

We have not made an attempt to optimize the constant $c$ obtained in the proof. The value we give is surely far from the truth.

In the other direction, there are easy examples of path-pairable planar examples with very large maximum degree; for example, consider the star $K_{1, n-1}$. Our second result finds an infinite family of path-pairable planar graphs with smaller (but of course still linear) degree.

\begin{restatable}{thm}{thmConstruction}\label{thm:construction}
There exist path-pairable planar graphs $G$ on $n$ vertices with $\Delta(G)=\frac{2}{3}n$.
\end{restatable}

Combining Theorems~\ref{thm:PPP} and~\ref{thm:construction}, we have that
\[
	10^{-10^{10}}n \leq \Delta^p_{\min}(n) \leq \frac{2}{3}n.
\]
However, there is currently a significant gap between the constants in the upper and lower bounds. Closing this gap and finding the truth is an interesting open problem. 

\subsection{Organization}
The remainder of the paper is organized as follows. In the next short section, we shall describe our construction establishing Theorem~\ref{thm:construction}. The third section of this paper contains a proof of our main theorem, Theorem~\ref{thm:PPP}. This proof relies on three preparatory lemmas and on some common facts concerning planar graphs. In particular, we use heavily the fact that any subset $X$ of the vertices of a planar graph induces less than $3|X|$ edges, and any bipartite planar graph on $n$ vertices has less than $2n$ edges. Finally, we close with some remarks and open problems.
\subsection{Notation}
Our notation is standard. Thus, for a graph $G$ and two subsets $X, Y \subset V(G)$ we say that a path in $G$ is an $X-Y$ path if 
it begins in $X$ and ends in $Y$. If $X = \{x\}$ and $Y = \{y\}$ are singletons, we shall simply say that the path is an $x-y$ path. 
For subsets $X, Y \subset V(G)$, $e(X, Y)$ is the number of edges with one endpoint in $X$ and the other in $Y$. As usual, $G[X]$ denotes the graph induced in $G$ with vertex set $X$.

\section{The Construction}
Our aim in this section is to prove Theorem~\ref{thm:construction}, which we restate here for convenience.
\thmConstruction*

\begin{proof}
Let $G$ be a graph on $n=6k$ vertices  with vertex set $V(G)=A\cup B \cup C \cup \{x_{AB},x_{BC}, x_{CA}\}$ where $|A|=|B|=|C|=2k-1$, and $x_{AB},x_{BC},x_{CA}$ denote three additional vertices forming a triangle such that $x_{AB},x_{BC},x_{CA}$ are joined to every vertex in $A\cup B$, $B\cup C$, and $C\cup A$, respectively. This graph is clearly planar. Let $\mathcal{P}$ be a pairing of the vertices and denote $u,v$ to be a pair of \emph{terminals} if $\{u,v\} \in \mathcal{P}$; we define the following pairing scheme depending on the position of the terminals:
\begin{enumerate}

\item If $\{u,v\}\subset \{x_{AB},x_{BC},x_{CA}\}$, join $u$ and $v$ by the  unique direct edge between them. 

\item If $u\in \{x_{AB},x_{BC},x_{CA}\}$ and $v\in A\cup B\cup C$ such that there exists a direct edge between $u$ and $v$, join them by this edge.

\item If $u\in \{x_{AB},x_{BC},x_{CA}\}$ and $v\in A\cup B\cup C$ such that there is no  edge between the terminals: we define a cyclic rotation $x_{AB}\rightarrow x_{BC}\rightarrow x_{CA}\rightarrow x_{AB}$ on the edges of the triangle formed by $x_{AB},x_{BC},x_{CA}$ and join $u$ and $v$ by a path of length $2$ by going along the directed edges. For example, if $u=x_{AB}$ and $v\in C$, we join the terminals by the path $ux_{BC}v$. The remaining cases can be dealt using the same pattern.

\item If $u,v\in A\cup B\cup C$ and they are in the same class, choose an arbitrary common neighbour (out of the two available) of $u$ and $v$ from $\{x_{AB},x_{BC},x_{CA}\}$ to join the terminals by a path of length 2. 

\item If $u,v\in A\cup B\cup C$ and they are in different classes, choose the unique common neighbor of $u$ and $v$ from $\{x_{AB},x_{BC},x_{CA}\}$ to join the terminals by a path of length 2.
\end{enumerate}

It is straightforward to check that the above instructions find edge-disjoint paths joining terminals, regardless of the choice of $\mathcal{P}$. 
\end{proof}

\section{The Proof of Theorem~\ref{thm:PPP}}

The aim of this section is to prove our main theorem, Theorem~\ref{thm:PPP}. 
Our prove is based on three preparatory lemmas. First, we shall introduce some terminology.
Let $G$ be a multigraph. We say that two multiedges $e, f$ of $G$ are at distance $d$ if the 
shortest path in $G$ joining an endpoint of $e$ and an endpoint of $f$ has length $d$. If two multiedges are at distance $0$, 
we shall simply say they are \emph{incident}.
Further, we shall refer 
to a matching of size $k$ as a $k$-\emph{matching}.
We say that a $k$-matching is \emph{good} if every pair of edges in the matching is at distance exactly $1$.
Notice that contracting all the edges of a good $k$-matching results in the complete graph $K_{k}$ (with potential multiple edges and loops).
\\
\begin{comment}
\begin{proof}
    \begin{align*}
    2^{-k}\frac{1+2^{-k-1}}{(1-2^{-k})^{2}} \le 2^{-k+1} \iff \\
    \frac{1+2^{-k-1}}{(1-2^{-k})^{2}} \le 2 \iff \\
    1+2^{-k-1} \le 2(1-2^{-k})^{2} \iff \\
    1+2^{-k-1} \le 2-2^{-k+2}+2^{-2k+1} \iff \\
    2^{-2k+1}+1-2^{-k-1}-2^{-k+2} \ge 0 \iff \\
    2^{-k-1}(2^{-k+2} -1) + 1 - 2^{-k+2} \ge 0 \iff \\
    (2^{-k+2}-1)(2^{-k-1}-1) \ge 0
    \end{align*}
\end{proof}
\end{comment}

Our first lemma says that in any multigraph either some multiedges `cluster' together or many pairs of multiedges are far apart, or one can find a good $k$-matching.
We shall need the following inequality.
\begin{fact}\label{fact}
    If $k \ge 2$ then $2^{-k}\left(\frac{1+2^{-k-1}}{(1-2^{-k})^{2}}\right) \le 2^{-k+1}$.
\end{fact}
The above inequality is easily seen to be equivalent to $(2^{-k+2}-1)(2^{-k-1}-1) \ge 0$.

\begin{lemma}\label{lem:multigraph}
    Let $k$ be a natural number and $\varepsilon_{1}, \varepsilon_{2}$ be positive reals such that $\varepsilon_{1} + \varepsilon_{2} \le 2^{-k}$.
    Then, for sufficiently large $M = M(k)$, if $G$ is a multigraph on $M$ multiedges, then at least one of the following conditions is satisfied.
    \begin{enumerate}
        \item \label{item_1} There is a multiedge in $G$ which is incident with at least $\varepsilon_{1} M$ multiedges;
        \item \label{item_2} There are at least $\varepsilon_{2} \binom{M}{2}$ pairs of multiedges which are at distance greater than $1$;
        \item \label{item_3} $G$ contains a good $k$-matching.
    \end{enumerate}
\end{lemma}
\begin{proof}
    We shall use induction on $k$.
    The base case when $k=1$ is trivial - Condition~\ref{item_3} is always satisfied.
    Assume then that $k \ge 2$ and the lemma is true for $k-1$.

    Suppose every multiedge is incident with at most $\varepsilon_{1} M$ multiedges and at most $\varepsilon_{2} \binom{M}{2}$ pairs of multiedges are at distance greater than $1$.
    We shall show that $G$ contains a good $k$-matching.
    By an averaging argument there is a multiedge $e$ which is at distance at most $1$ from at least $(1-\varepsilon_{2})M-1$ multiedges.
    Let $E'$ be the set of those multiedges which are at distance exactly $1$ from $e$.
    It follows from our assumptions that $M' = |E'| \ge (1-\varepsilon_{1}-\varepsilon_{2})M-1 \ge (1-2^{-k})M-1$.
    Let $G'$ be the multigraph spanned by $E'$.
    By assumption, at most $\varepsilon_{2}\binom{M}{2}$ of the multiedges in $G'$ are at distance greater than $1$.
    Therefore, since $M \le \frac{M'+1}{1-2^{-k}}$, for large enough $M$ (and hence large enough $M'$) we have that at most
\begin{align*}
\varepsilon_{2}\binom{M}{2} \le \varepsilon_{2}\binom{\frac{M'+1}{1-2^{-k}}}{2}
&= \frac{\varepsilon_{2}}{(1-2^{-k})^{2}}\left(1+\frac{1}{M'}\right)\left(1+\frac{1+2^{-k}}{M'-1}\right)\binom{M'}{2} \\
&\le \frac{\varepsilon_{2}(1+2^{-k-1})}{(1-2^{-k})^{2}} \binom{M'}{2},
\end{align*}
pairs of multiedges in $G'$ are at distance greater than $1$. Note that for $k \ge 2$ one has
\begin{align*}
\varepsilon_{1}+\varepsilon_{2}\frac{1+2^{-k-1}}{(1-2^{-k})^{2}} &\le \varepsilon_{1}\frac{1+2^{-k-1}}{(1-2^{-k})^{2}}+\varepsilon_{2}\frac{1+2^{-k-1}}{(1-2^{-k})^{2}}\\
&\le 2^{-k}\frac{1+2^{-k-1}}{(1-2^{-k})^{2}} \leq 2^{-(k-1)},
\end{align*}
where the last inequality is precisely Fact~\ref{fact}. Therefore, by the induction hypothesis, $G'$ contains a good $(k-1)$-matching.
    But since $e$ is at distance $1$ from any multiedge in $G'$, we also have a good $k$-matching in $G$.
    \end{proof}     
    
    %= \frac{\varepsilon_{2}}{(1-2^{-k})^{2}}\binom{M'}{2}(1+\frac{1}{M'})(1+\frac{1+2^{-k}}{M'-1}) \le \frac{\varepsilon_{2}(1+2^{-k-1})}{(1-2^{-k})^{2}} \binom{M'}{2}$ \todo{I'll try to simplify this later} pairs of multiedges in $G'$ are at distance greater than than $1$.
    %For $k \ge 2$ it follows from some straightforward calculations that $\varepsilon_{1}+\varepsilon_{2}\frac{1+2^{-k-1}}{(1-2^{-k})^{2}} \le \varepsilon_{1}\frac{1+2^{-k-1}}{(1-2^{-k})^{2}}+\varepsilon_{2}\frac{1+2^{-k-1}}{(1-2^{-k})^{2}} \le 2^{-k}\frac{1+2^{-k-1}}{(1-2^{-k})^{2}} \le 2^{-(k-1)}$ \todo{simplify\dots}.

Since we shall be operating with planar graphs, we single out the following corollary.
\begin{corollary}
\label{cor_multiedges}
    Let $M$ be a sufficiently large integer and let $\varepsilon_{1},\varepsilon_{2}$ be positive reals such that $\varepsilon_{1}+\varepsilon_{2} \le \frac{1}{32}$.
    If $G$ is a planar multigraph with $M$ multiedges then either $G$ has a multiedge which is incident with at least $\varepsilon_{1} M$ multiedges or there are at least $\varepsilon_{2} \binom{M}{2}$ pairs of multiedges at distance greater than $1$.
\end{corollary}
\begin{proof}
If $G$ contained a good $5$-matching then it would contain a $K_5$ minor. 
\end{proof}

One strategy in the proof of our main theorem is to consider a suitable bipartition of our path-pairable planar graph, and to 
exploit the fact that any bipartite planar graph on $n$ vertices has at most $2n - 4$ edges. To exploit this last property we 
shall need ways of finding pairings of the vertices such that their corresponding edge-disjoint paths contribute `many' edges 
to the bipartition. This is formalized in the following lemma.

\begin{lem}
    \label{lem_main_ppp}
    Let $D$ be an integer and $0<\varepsilon \leq 1/2$.
    Then there exists $c>0$ such that the following is true.
    Suppose $G$ is a path-pairable planar graph on $n$ vertices with $\Delta = \Delta(G) \le cn$.
    Let $A, U \subset V(G)$ be given with $U \subset A$ such that every vertex in $A$ has degree at most $D$, $|A| \ge (1-\varepsilon)n$ and $|U| \ge \varepsilon n$.
    Let $B = V(G) \setminus A$.
    Then there is a pairing of the vertices in $U$ which contributes to
    at least $2|U| - 16\varepsilon n$ edges between $A$ and $B$.
\end{lem}
\begin{proof}
We say that a path in $G$ is \emph{weak} if it begins and ends in $A$, uses no edges inside $B$, 
and uses at most $2$ edges between $A$ and $B$. 
Now, let $C := \lceil{4\eps^{-1}}\rceil$ and note that since $\eps \leq 1/2$ we have that $\frac{3}{C-2} \leq \varepsilon$. 
For every $x \in U$, let 
$X(x, C) = \{ u \in U: \dist(u, x) \leq C\}$ and $Y_x = \{ u \in U: \exists$ \text{ a weak } $x-u$ \text{ path in } $G\}$. 
Finally, consider the set $U_x = X(x,C)\cap Y_x$. We claim that $U_x$ is small for every $x \in U$; namely, it is 
easy to see that 
$|U_x| \leq D^C + D^CD\Delta D^C = D^C\left(1 + D^{C+1}\Delta\right)$.
Choose $c = c(D, \varepsilon) = \frac{\eps}{4D^{2C+1}}$ so that $\Delta \le cn$.
Then $|U_x| \leq D^C\left(1 + D^{C+1}\Delta\right) \le \left(D^C+D^{2C+1}\right)\Delta \leq \frac{\varepsilon}{2}n$.

Let us define an auxiliary graph $G_U$ with vertex set $U$ where we join two vertices 
$x, y$ provided $y \notin U_x$ (equivalently, $x \notin U_y$). It is easy to see that $G_U$ has 
a perfect matching (or `almost' perfect, if $|U|$ is odd; this makes no difference for us). Indeed, 
the degree of every vertex in $G_U$ is at least $|U| - \frac{\eps}{2}n \geq |U|/2$, and 
therefore $G_U$ has a Hamilton cycle. Fix a perfect matching $\mM$ in $G_U$ according to this Hamilton 
cycle and fix a pairing $\mP$ of the vertices of $G$ where each edge of $\mM$
forms a pair. Finally, since $G$ is path-pairable, choose a collection of edge-disjoint paths $\mR$ that realize this pairing. Observe that any path from $\mR$ must use an even number of edges between $A$ and $B$.
There are two types of edges in $e = xy \in \mM$ with respect to this realization: either the $x-y$ 
path in $\mR$ is weak but $\dist(x, y) > C$, or this $x-y$ path uses at least $4$ edges between $A$ and $B$. 
Let $\mM = \mE_0 \cup \mE_1 \cup \mE_{2}$, where $\mE_0$ denotes the edges satisfying 
the former condition, $\mE_1$ the latter, and $\mE_2$ denotes the remaining edges.
We claim that most edges are in $\mE_1$.
Indeed, observe that if $e = xy \in \mE_2$, then the $x-y$ path must use edges from $B$.
By planarity we have $e(B) < 3|B|$, and therefore $|\mE_2| < 3\varepsilon n$.
Using planarity again we have that $e(A) < 3|A|$. 
On the other hand, for each edge in $\mE_0$ its path in $\mR$ uses more than $C$ edges, at most $2$ of which are 
in the cut $\{A, B\}$, and none of which belong to $B$. Accordingly, since these paths are edge-disjoint, we have that 
$e(A) \geq (C-2)|\mE_0|$ and so 
\[
	|\mE_0| < \frac{3}{C-2}|A| \leq \eps|A|.
\]

Therefore, $|\mE_1| \geq \frac{1}{2}|U| - \eps|A| - 3\varepsilon n \geq \frac{1}{2}|U|-4\varepsilon n$.
It follows that since every path in $\mR$ pairing an edge in $\mE_1$ contributes at least $4$ edges between $A$ and $B$, and these 
paths must be edge-disjoint, we have 
\[
	e(A, B) > 2|U|-16\varepsilon n.
\]
This completes the proof of Lemma~\ref{lem_main_ppp}.
\end{proof}

Our final lemma allows us to quantify more precisely the degree distribution in any bipartite planar graph.

\begin{lem}
    \label{lem_deg_2}
    Let $G$ be a bipartite planar graph on $n$ vertices with parts $A$, $B$, and let $A' \subset A$ be the set of vertices in $A$ with degree at least $3$.
    Then the following are true.
    \begin{enumerate}
        \item The number of vertices in $A$ with degree exactly $2$ is at least $e(A, B) - n - 3|B|$;
        \item $|A'| < 2|B|$;
        \item $e(A', B) < 6|B|$.
    \end{enumerate}
\end{lem}
\begin{proof}
For each $i\geq 0$ let $A_i, A_{\leq i},$ and $A_{\geq i}$ denote the number of vertices in $A$ that have degree $i$ in $G$, degree at most $i$, and degree at least $i$, respectively.
Because of planarity we have that $e(A', B) < 2(|A'| + |B|)$.
Alternatively, $e(A', B) \geq 3|A'|$ so it follows that $A_{\geq 3} = |A'| < 2|B|$, and so
$e(A', B) \le 2(|A'| + |B|) < 6|B|$, establishing the second and third items.
Further, we can bound the number of edges between $A$ and $B$ as
\begin{align*}
    e(A, B) &\leq A_{\leq 1} + 2(|A| - A_{\leq 1} - A_{\ge 3}) + e(A', B) \\
&\leq A_{\leq 1} + 2(|A| - A_{\leq 1} - 2|B|) + 6|B| \\
&\leq 2|A| - A_{\leq 1} - 4|B| + 6|B| \\
&\leq 2|A| - A_{\le 1} + 2|B|.
\end{align*}
It follows that $A_{\le 1} \le 2|A| + 2|B| - e(A, B)$.
Finally, we see that $A_2 = |A| - |A'|-A_{\le 1} > e(A, B) - |A| - 4|B| = e(A, B) - n - 3|B|$, as required.

\end{proof}

We are now in a position to prove our main theorem. First, let us give a rough sketch of the proof. Let $G$ be a 
path-pairable planar graph. We first partition the vertex set of $G$ into the set $A$ of vertices of small degree 
and the set $B$ of vertices of large degree. We can apply Lemma~\ref{lem_main_ppp} to find that there are many 
edges in this cut. We shall then show that most vertices in $A$ have degree $2$ in this bipartite graph.
If $Y \subset A$ denotes the vertices of degree $2$, then we define a planar multigraph with vertex set $B$ where we join $x, y \in B$
whenever there is a $v \in Y$ joined to precisely $x$ and $y$. Now, using Corollary~\ref{cor_multiedges}, we are able to either find 
a vertex of linear degree in $B$, or we can find many pairs of multiedges in our multigraph that are far apart. This, however, allows us 
to find a pairing which contributes to more than $2n$ edges between $A$ and $B$, a contradiction to planarity. 

We restate Theorem~\ref{thm:PPP} for convenience.
\thmPPP*

\begin{proof}
Suppose $G$ is a path-pairable planar graph and
fix some large constant $D$ so that $D^{-1} \leq 8.5\cdot 10^{-6}$. Partition the vertex set of $G$
into sets $A$ and $B$, where $B = \{v \in V(G): d(v) \geq D\}$ and $A = V(G) \setminus B$. Since 
$e(G) < 3n$ it easily follows that $|B| \leq 6D^{-1}n : = \eps n$. Suppose that $\Delta(G) < cn$, where 
$c$ is sufficiently small (depending only on $D$) given by Lemma~\ref{lem_main_ppp}. More precisely, 
we may take $c = \frac{\eps}{4D^{2\lceil{4/\eps}\rceil + 1}}$.

Our aim is to obtain a contradiction to the planarity of $G$, and so there must exist a vertex of degree at least $cn$.
Of course, this is trivial if $cn \leq 1$, so we shall assume throughout that $n \ge 1/c$.
By Lemma~\ref{lem_main_ppp} (with $U = A$) we have that there are at least $2|A|-16\varepsilon n \geq 2n-18\varepsilon n$ edges between $A$ and $B$.

Next, we shall show that there is a large subset of $A$ which induces a graph with maximum degree at most $2$.
To see this, let $A_{0} = A, B_{0} = B$.
Suppose $A_{i}, B_{i}$ have been defined already. If there is a vertex $v \in A_{i}$ such that $d_{A_{i}}(v) > d_{B_{i}}(v)$,
then let $A_{i+1} = A_{i} \setminus \left\{ v \right\}$ and $B_{i+1} = B_{i} \cup \left\{ v \right\}$. 
Notice that $e(A_{i+1}, B_{i+1}) \ge e(A_{i}, B_{i}) + 1$, and so $e(A_{i+1}, B_{i+1}) \ge e(A, B) + i \ge 2n - 18\varepsilon n+ i$.
Let $t \ge 0$ be such that there is no $v \in A_{t}$ with more neighbours in $A_{t}$ than in $B_{t}$.
Observe that $t \le 18\varepsilon n$ (otherwise $e(A_{t}, B_{t}) \ge 2n$), and accordingly $|B_{t}| = |B| + t \le \varepsilon n + 18\varepsilon n = 19\varepsilon n$.

Let $X \subset A_t$ be the set of vertices in $A_{t}$ with at least $3$ neighbours in $A_{t}$.
Since every vertex in $A_{t}$ has more neighbours in $B_{t}$ than in $A_{t}$, we have that every vertex in $X$ has at least $3$ neighbours in $B_{t}$.
Therefore, by Lemma~\ref{lem_deg_2}, $|X| \le 2|B_{t}|$, $e(X, B_{t}) \le 6|B_{t}|$, and there are at least $e(A_{t}, B_{t}) - n - 3|B_{t}| \ge e(A, B) - n - 3|B_{t}|$ vertices in $A_{t}$ with exactly two neighbours in $B_{t}$. 
Let $A^* = A_{t} \setminus X$ and $B^* = B_{t} \cup X$.
Now we have that every vertex in $A^*$ has at most $2$ neighbours in $A^*$ and $|B^*| \le 3|B_{t}| \le 57 \varepsilon n$, so $|A^*| \ge n - 57 \varepsilon n$.
We have to make sure we still have many vertices in $A^{*}$ with exactly two neighbours in $B^{*}$.
Notice that if a vertex $v \in A_{t}$ had two neighbours in $B_{t}$ and was not adjacent to any vertex in $X$ then $v \in A^{*}$ and $v$ still has exactly two neighbours in $B^{*}$.
Therefore we only have to worry about the vertices in $A_{t}$ which are adjacent to some vertices in $X$.
Observe that $e(X, A^*) \le e(X, B_{t}) \le 6|B_{t}|$, and so there are at least $e(A, B) -n - 9|B_{t}| 
\ge (2n - 18\eps n)- n - 9\cdot19\eps n = n - 189\eps n$ vertices in $A^{*}$ with exactly $2$ neighbours in $B^{*}$.
Hence there are at most $189\varepsilon n$ vertices in $A^{*}$ which do not have degree $2$ in $B^{*}$.

We say that an edge $uv \in G$ is \emph{bad} if one of the followings holds:
\begin{enumerate}
    \item(Type I) $uv \in G[B^*]$.
    \item(Type II) $uv \in G[A^*]$ and $u$ (or $v$) has degree not equal to $2$ in $B^*$.
    \item(Type III) $uv \in G[A^*]$, $d_{B*}(u)=d_{B*}(v) =2$, and $N_{B^*}(u) \neq N_{B^*}(v)$.
    \item(Type IV) $uv \in G$, such that $u \in A^*, v \in B^*$, and $d_{B^{*}}(u) \ge 3$.
\end{enumerate}

We have the following bound on the number of bad edges.
\begin{claim}
    There are at most $1233 \varepsilon n$ bad edges.
\end{claim}
\begin{proof}
    We are going to bound the number of bad edges of each type.

    Note that by planarity, there are at most $3|B^*|$ edges in $B^*$ so there are at most $3|B^*| \le 171 \varepsilon n$ edges of
    Type I.

    Now, since every vertex in $A^{*}$ has at most two neighbours in $A^{*}$, each vertex in $A^{*}$ with degree not equal to $2$ in $B^{*}$ contributes to at most two bad edges of Type II. As there are at most $189\eps n$ vertices in $A^*$ which do not have degree 
    $2$ in $B^*$, it follows that there are at most $378\varepsilon n$ bad edges of Type II.

    Let us consider bad edges of Type III. Since $G[A^*]$ has maximum degree $2$, we can partition the edges of $G[A^*]$ into at most $3$ matchings, $M_{1}, M_{2}, M_{3}$.
    It is well known (and easy to see) that contracting an edge in a planar graph preserves planarity.
    It follows that, for $i \in \left\{ 1, 2, 3 \right\}$, we can contract the edges of $M_{i}$ while still preserving planarity.
    Denote this new graph by $\tilde{G_{i}}$ with vertex set $\tilde{A_{i}} \cup B^{*}$.
    Since $\tilde{G_{i}}$ is planar, from Lemma~\ref{lem_deg_2} we have that there are at most $2|B^{*}|$ vertices in $\tilde{G_{i}}$ with at least $3$ neighbours in $B^{*}$.
    Therefore, at most $2|B^{*}|$ edges in $M_{i}$ can be bad of Type III.
    Hence, there are at most $6|B^{*}| \le 342\varepsilon n$ bad edges of Type III.

    Finally, by Lemma~\ref{lem_deg_2} there can be at most $6|B^*| \le 342 \varepsilon n$ bad edges of Type IV.

    So in total there are at most $1233\varepsilon n$ bad edges of any type.
\end{proof}

Let $Y \subseteq A^{*}$ be the set of vertices with degree exactly $2$ in $B^{*}$.
We now define an auxiliary \emph{multi}graph $G_{B^*}$ in the following way. The vertex set of $G_{B^*}$ is $B^{*}$ and for 
any two vertices $x, y \in B^{*}$, join $x$ to $y$ by an edge for every $v \in Y$ that is joined precisely to $x$ and $y$.

\begin{claim}
    $G_{B^*}$ is planar.
\end{claim}
\begin{proof}
    This is clear since the bipartite graph $G[Y, B^*]$ between $Y$ and $B^*$ is planar, and contracting edges 
    preserves planarity. 
\end{proof}

%We shall need the following lemma, as stated (as a corollary) in the Overleaf document.
%Note that when we talk about adjacency and distance between \emph{edges} in the following, we 
%mean in the corresponding line graph.

%\begin{lem}\label{lem:multiedges}
%Let $M$ be a sufficiently large positive integer and let $\eps_1, \eps_2 > 0 $
%with $\eps_1 + \eps_2 < 1/32$. If $G$ is a planar multigraph with $M$ edges, then 
%either of the following must occur:
%\begin{enumerate}
%\item There exists an edge adjacent to at least $\eps_1M$ edges.
%\item At least $\eps_2\binom{M}{2}$ pairs of edges are at distance greater than $1$.
%\end{enumerate}
%\end{lem}

Apply Corollary~\ref{cor_multiedges} to the multigraph $G_{B^*}$ with $\eps_1 = \eps_2 = 1/100$.
Notice that if an edge in $G_{B^*}$ has degree bigger than $\frac{1}{100}|Y|$ then one of its endpoints has degree at least $\frac{{1}}{200}|Y|$. However, recall that we initially assumed $\Delta(G) < cn$, and certainly $c \leq 1/400$ by our choice of $D$. Accordingly, since $|Y| \geq n - 189\eps n 
\ge n/2$, we obtain a vertex of degree at least 
\[
	2c|Y| \ge cn,
\]
a contradiction. 

So we may assume that there are at least $\frac{1}{100} \binom{|Y|}{2}$ pairs of edges in $G_{B^*}$ which are at distance greater than $1$. Note that if $H$ is any graph on $n$ vertices with edge density at least $\delta$, then it is easy to greedily 
find a matching of size at least $\frac{\delta}{10}n$.
It follows that we may select a collection of pairwise disjoint pairs $\mathcal{P}$ in $Y$, such that $|\mathcal{P}| \geq \frac{1}{1000}|Y|\ge \frac{1}{2000}n$, and such that for every $\left\{ u, v \right\} \in \mathcal{P}$, their corresponding edges in $G_{B^*}$ are at distance greater than $1$.

We need the following two claims.
\begin{claim}\label{claim1}
    \label{claim_path_in_A}
    Let $P$ be a path contained in $A^{*}$ which has at least two vertices and does not contain any bad edges.
    Then every vertex $v \in P$ has the same neighbourhood (of size $2$) in $B^{*}$.
\end{claim}
\begin{proof}
   This is immediate from the definiton of a bad edge.
\end{proof}

\begin{claim}
    Let $u, v \in Y$ be two vertices whose corresponding edges in $G_{B^*}$ are at distance greater than $1$. 
    Then any path in $G$ joining $u$ and $v$ either contains some bad edges, or uses at least $6$ edges between $A^{*}$ and $B^{*}$.
\end{claim}
\begin{proof}
    Suppose $P$ is a path joining $u$ and $v$ which does not use any bad edges. By defintition and using claim \ref{claim1}, all vertices of $V(P)\cap A^{*}$ are in $Y$, it can not have an edge inside $B^{*}$ and it must use $2$ or $4$ edges between $A^{*}$ and $B^{*}$. We may assume $P$ uses $4$ edges as the other case follows from the same argument. Let $P=P_1{e_1}{e_2}P_2{e_3}{e_4}P_3$, where $\{e_1,e_2,e_3,e_4\}$ are edges between $A^{*}$ and $B^{*}$ and $P_1,P_2,P_3$ are paths inside $Y$. From claim \ref{claim1} appplied to $P_1$, $P_2$ and $P_3$ we deduce that the edge of $u$ in $G_{B^{*}}$ is at distance at most $1$ to the edge of $v$ in $G_{B^{*}}$. 
\end{proof}

The proof of Theorem~\ref{thm:PPP} is nearly complete. Indeed, since $G$ is path-pairable, there are edge-disjoint paths joining every pair of $\mathcal{P}$, and hence the pairs in $\mathcal{P}$ contribute to at least $6(|\mathcal{P}| - 1233\varepsilon n)$ edges between $A^{*}$ and $B^{*}$.

Let $P$ be the union of the vertices in $\mathcal{P}$ and let $U = A^{*} \setminus P$.
Suppose first that $|U| < 57\eps n$. It follows that
\[ 
	2|\mP| > (n - 57\eps n) - 57\eps n, 
\]
so $|\mP| > n/2 - 57\eps n$. Then the above pairing contributes at least $6(n/2 - 1290\eps n) = 3n - 7740\eps n$ edges between $A^*$ 
and $B^*$. But this is at least $2n$ whenever $\eps \le 7740^{-1}$ which is guaranteed by our choice of $D$, a contradiction. 
Therefore, we may assume that $|U| \ge 57\eps n$.
By Lemma~\ref{lem_main_ppp} (since $c$ is small enough) there is a pairing of the vertices in $U$ which contributes to at least $2|U| - 16\cdot57\eps n = 2|U| - 912\eps n$ edges between $A^{*}$ and $B^{*}$.
Hence in total the number of edges between $A^*$ and $B^*$ is
\begin{align*}
&\ge 6(|\mathcal{P}| - 1233\varepsilon n) + 2|A^{*}| - 4|\mathcal{P}| - 912\eps n \\ 
&\ge 2|\mP| + 2(n -57\eps n) - 6\cdot1233\eps n - 912\eps n \\
&\ge 2n + n/1000 - 8424\eps n.
\end{align*}
So by our choice of $D$ we get that $8424\eps \leq \frac{1}{1000}$, and so there are at least $2n$ edges between $A^*$ and 
$B^*$, a contradiction to the planarity of $G$. It follows that there must exist a vertex of degree at least $cn$. 

\end{proof}

\section{Final Remarks and Open Problems}

It is worth observing that we only used that our graph did not contain a $K_{3,3}$ minor rather than the full planarity condition. In particular, note that our proof relied heavily on the fact that number of edges in any bipartite planar graph on $n$ vertices is less than $2n$. This constraint on the number of edges also holds for bipartite graphs which contain no $K_{3,3}$ minor, which allows us to carry through the rest of the proof. 

We conjecture the following: 

\begin{conjecture}
For any $t$ there exists a constant $c=c(t)$ such that every path-pairable graph on $n$ vertices without a $K_t$ minor must contain a vertex of degree at least $cn$. 
\end{conjecture}

Finally, recall that we defined $\Delta_{\min}^p(n)$ to be the minimum of $\Delta(G)$ over all $n$-vertex path-pairable planar graphs $G$. We have shown that $\Delta_{\min}^p(n)$ grows linearly in $n$; however, as mentioned in the Introduction, the constants in the upper and lower bounds are quite far apart.  We close with the following problem.

\begin{problem}
Determine $\Delta^{p}_{\min}(n)$ for sufficently large $n$.
\end{problem}

We do not know if our construction yielding the upper bound of $2n/3$ is optimal, and a significant improvement on our lower bound would be very interesting.

\bibliographystyle{acm}
\bibliography{main.bib}
\end{document}